\documentclass[11pt]{amsart}

\usepackage{amsthm, amsfonts, mathrsfs}
\usepackage{amssymb, amsmath, theoremref}
\usepackage{times, txfonts}
\usepackage{mathtools}

\usepackage{todonotes}

\newcommand{\R}{{\mathbf{R}}}
\newcommand{\T}{{\mathbf{T}}}

\newcommand{\pr}{\mathop{\! \, \rm pr }\nolimits}

\newcommand{\proj}{\mathrm{pr}}

\DeclarePairedDelimiterX\inner[2]{\langle}{\rangle}{#1 , #2}

\theoremstyle{plain}
\newtheorem{claim}{\sc Claim}[section]
\newtheorem{corollary}[claim]{\sc Corollary}

\newtheorem{theorem}[claim]{\sc Theorem}

\theoremstyle{definition}

\theoremstyle{remark}

\begin{document}

\title[]{An exotic symplectic $\R^6$}
\author[L. Bates and M. Melko]{Larry M. Bates and O. Michael Melko}

%\date{\today }

\begin{abstract}
An explicit example of an exotic symplectic $\R^6$ is given.  Together with 
an earlier known example on $\R^4$, this yields an explicit exotic 
symplectic form on $\R^{2n}$ for all $n\geq2$.
\end{abstract}

\maketitle

\section{Introduction}
Darboux' theorem states that there is a symplectic chart in the neighbourhood 
of any point in a symplectic manifold.  This means that the symplectic form 
$\omega$  of any symplectic manifold  locally looks like the canonical 
symplectic form $\omega_0$ on $\R^{2n}$, that is, there are local coordinates 
$(x_1,\dots,x_n,y_1,\dots,y_n)$ in which $\omega$ has the form
\[
  \omega = \sum_{a=1}^n dx_a\wedge dy_a.
  \]
A symplectic structure on $\R^{2n}$ is \emph{exotic} if such a Darboux chart cannot be extended to a global symplectic embedding.
It was of considerable interest when Gromov proved the existence of such a 
structure in \cite{gromov85} by first proving a nonvanishing 
theorem. More precisely, he proved the following. 
\begin{theorem}
  For a compact embedded Lagrangian submanifold $L$ in the standard symplectic 
structure $(\R^{2n},\omega_0)$, the relative class $[\omega_0]\in 
H^2(\R^{2n},L,\R)$ never vanishes.
\end{theorem}
He then gave an existence proof for a symplectic form $\omega$ on $\R^4$ with 
an exact Lagrangian torus (\,the class $[\omega]=0\in 
H^2(\R^{2n},L,\R)$\,), and hence an exotic symplectic structure.  
Later, Muller proved the existence of a symplectic structure on $\R^6$ with an 
embedded Lagrangian three-sphere \cite{muller}, which, by Gromov's theorem, must 
be exotic. Similarly, Cho and Yoon have shown the existence of a 
symplectic $\R^4$ with a compact exact Lagrangian submanifold of genus two 
\cite{cho-yoon}. Unfortunately, the proofs of these theorems are 
non-constructive.  It is therefore of interest to be able to write such a 
structure down explicitly.  An example on $\R^4$ was given in 
\cite{bates-peschke}. In this note, we consider the case of $\R^6$.

\section{The construction}

An example of an exotic symplectic $\R^6$ will be constructed by 
giving the symplectic potential. Let 
$(x_1,x_2,x_3,y_1,y_2,y_3)$ be coordinates on $\R^6$. Consider the one forms 
\[
  \psi_a := 
\frac12\left(1-\frac76\left(x_a^2+y_a^2\right)+\frac16\left(x_a^2 
+y_a^2\right) ^ 2 \right) \left(x_a\,dy_a - y_a\,dx_a\right) \quad a=1,2,3
\]
and set
\[ 
\psi :=\psi_1 + \psi_2 +\psi_3.
  \]
  Observe that $\psi\equiv0$ on the three torus $T$ defined by
  \[
    \left\{x_1^2+y_1^2=1\right\} \cap  \left\{x_2^2+y_2^2=1\right\} 
\cap  \left\{x_3^2+y_3^2=1\right\}.
  \]
Define the function $h$ by
\[
  h=\frac12\left(x_1^2 + x_2^2 +x_3^2+y_1^2 + y_2^2 +y_3^2\right),
\]
and let $S=h^{-1}(3/2)$ be the sphere of radius $\sqrt{3}$ centered at the 
origin.
\begin{claim} \label{rank-claim}
  The form $ dh\wedge d\psi \wedge d\psi$ has rank five on the sphere 
$S$.
\end{claim}
\begin{proof}
  Verification of the claim is summarized in the appendix.
\end{proof}

Note that because the sphere $S$ is a level set of $h$, this 
implies that the pullback of $d\psi\wedge d\psi$ to $S$ has rank 
four everywhere on $S$. Define, in a neighbourhood of $S$, the one form $\chi$ 
by
\[
  \chi := \ast \,(dh\wedge d\psi \wedge d\psi),
\]
where $\ast$ is the Hodge star operator in $\R^6$. Note that 
$\chi\neq0$ in a neighbourhood of $S$.  Set $s=\sqrt{2h}-\sqrt{3}$, so that 
$S=\{s=0\}$.  This allows the identification of the normal bundle 
$\pr:N\rightarrow S$ of the sphere $S$ with a neighbourhood of the sphere.  
Finally, define a two form $\omega$ by
\[
  \omega := d(\proj^*(\psi|_S) + s\chi).
\]

\begin{theorem}
  The two form $\omega$ is a symplectic form in a neighbourhood of $S$.  
Furthermore, the torus $T$ is an exact Lagrangian submanifold. 
\end{theorem}
\begin{proof}
  The form $\omega$ is closed because it is the exterior derivative of a one 
form, and it is nondegenerate because
\[
  \omega\wedge\omega\wedge\omega = 3 \,ds\wedge \chi\wedge d\psi\wedge d\psi 
+O(s)
\]
in a neighbourhood of $S$.  Furthermore, since $T\subset S$ and 
$\psi|_{T}\equiv0$, $T$ is an exact Lagrangian submanifold.

\end{proof}
\begin{corollary}
  The two form $\omega$ defines an exotic symplectic structure on $\R^6$.
\end{corollary}
\begin{proof}
  A neighbourhood of $S$ is identified with the normal bundle of $S$, and by 
removing a point of $S$ not in the torus $T$, and its normal fibre, we have a 
manifold diffeomorphic to $\R^6$ with an exact Lagrangian submanifold.
\end{proof}

As a final note, we point out that any direct product of the exotic structures on $\R^4$ and $\R^6$ described in \cite{bates-peschke} and here is exotic. Since any integer $n \geq 2$ can be written as $n = 2m + k$ with $m \geq 0$ and $k = 0\ {\rm or}\ 3$, we can bootstrap these exotic structures to an exotic structure on $\R^{2n}$ for any $n\geq 2$.

\bibliographystyle{plain}

\section{appendix}

The object of this appendix is to verify Claim \ref{rank-claim}. First of all, the formula for the one forms $\psi_a$ was inspired by the overtwisted contact structures of Eliashberg \cite{eliashberg}. The coefficient is based on the polynomial $p$, which satisfies
\[
  p(x):=1-7x/6 +x^2/6,\quad p(0)=1, \quad p(1)=0, \quad p(3)=-1.
\]

The actual expression for the five form $dh\wedge d\psi \wedge d\psi$ is far 
too complicated to reproduce here, but the verification of the rank condition 
in Claim \ref{rank-claim} can be simplified by observing that the entire 
construction is invariant under the action of the torus group $\T^3$, which acts 
by rotation in each $(x_a, y_a)$ plane.  This implies that we may set all 
$y_a=0$ in the formulae, reducing the problem to verifying that there is no 
simultaneous zero of the three polynomials
\[
	x_3 q\left(x_1^2, x_2^2\right), \quad
	x_1 q\left(x_2^2, x_3^2\right), \quad
	x_2 q\left(x_3^2, x_1^2\right), 
\]
when restricted to the sphere $x_1^2 + x_2^2 +x_3^2 =3$, where $q$ is 
defined by the formula
\[
  q(x,y) := 2-\frac{14}{3}(x+y)+ \frac{205}{18}xy +x^2+y^2 
-\frac{7}{3}(x^2y-x\,y^2).
\]
It is easily seen that there are no simultaneous zeros of the three polynomials 
above with $x_1, x_2,\ {\rm or}\ x_3 = 0$. The nonexistence of a simultaneous 
zero is therefore equivalent to showing that
\[
Q(x,y,z) := q(x,y)^2 + q(y,z)^2 + q(z,x)^2
\]
has a positive minimum on the simplex $\sigma = \{(x,y,z) : x+y+z =3;\ x,y,z \geq 0 \}$.
It was found by means of \emph{Mathematica} that $Q$ is bounded below by $7\times 10^{-3}$ in the ball of radius $3$ centered at the origin. Since $\sigma$ is contained in this ball, this verifies the claim.

\vspace{20pt}

\begin{minipage}[t]{0.5\textwidth}
\noindent Larry M. Bates \newline
Department of Mathematics \newline
University of Calgary \newline
Calgary, Alberta \newline
Canada T2N 1N4 \newline
bates@ucalgary.ca
\end{minipage}
\begin{minipage}[t]{0.5\textwidth}
\noindent Mike Melko \newline
Mathematical Synergistics LLC \newline
19 7th Ave. S.E., Suite 208\newline
Aberdeen, SD 57401 \newline
USA\newline
mike.melko@gmail.com
\end{minipage}

\end{document}